\documentclass[a4wide,12pt]{amsart}  
\usepackage{amsmath,amssymb,amsthm}     
\usepackage[utf8]{inputenc}
\usepackage{enumitem}

\def\dd{\mathrm d}

\theoremstyle{plain}      
 
\newtheorem{thm}{Theorem}

\newtheorem{prop}[thm]{Proposition}

\theoremstyle{definition}

\newtheorem{remark}[thm]{Remark}

\DeclareMathAlphabet{\doba}{U}{msb}{m}{n}         
\gdef\mC{\doba{C}}
\gdef\mH{\doba{H}}
\gdef\mN{\doba{N}}

\gdef\mR{\doba{R}}

\gdef\mZ{\doba{Z}}

\def\SL{\mathrm{SL}}

\def\di{{\rm d}}

\let\<\langle 
\let\>\rangle 
\let\ti\tilde

\newcommand{\definedas}{\mathrel{\raise.095ex\hbox{\rm :}\mkern-5.2mu=}}

\title[On Weyl-reducible conformal manifolds]{On Weyl-reducible conformal manifolds and lcK structures}
\author{Farid Madani, Andrei Moroianu, Mihaela Pilca}

\address{Farid Madani\\Fakult\"at f\"ur Mathematik\\
Universit\"at Regensburg\\Universit\"atsstr. 31 
D-93040 Regensburg, Germany}
\email{madani@mathematik.uni-regensburg.de}

\address{Andrei Moroianu \\ Laboratoire de Math\'ematiques d'Orsay, Univ. Paris-Sud, CNRS, 
Universit\'e Paris-Saclay, 91405 Orsay, France }
\email{andrei.moroianu@math.cnrs.fr}

\address{Mihaela Pilca\\Fakult\"at f\"ur Mathematik\\
Universit\"at Regensburg\\Universit\"atsstr. 31 
D-93040 Regensburg, Germany}
\email{mihaela.pilca@mathematik.uni-regensburg.de}

\begin{document}

\maketitle

\begin{abstract}
	A recent result of M. Kourganoff states that if $D$ is a closed, reducible, non-flat Weyl connection on a compact conformal manifold $M$, then the universal cover of $M$, endowed with the metric whose Levi-Civita covariant derivative is the pull-back of $D$, is isometric to $\mR^q\times N$ for some irreducible, incomplete Riemannian manifold $N$. Moreover, he characterized the case where the dimension of $N$ is $2$ by showing that $M$ is then a mapping torus of some Anosov diffeomorphism of the torus $\mathbb T^{q+1}$. We show that in this case one necessarily has $q=1$ or $q=2$. 
\end{abstract}
	
\section{Weyl-reducible manifolds}

 Let $(M,c)$ be a compact conformal manifold. A {\em Weyl structure} on $M$ is a torsion-free linear connection $D$ preserving the conformal structure $c$, in the sense that for every Riemannian metric $g\in c$, $D_Xg=\theta_g(X)g$ for some $1$-form $\theta_g$ on $M$ called the {\em Lee form} of $D$ with respect to $g$. If $g':=e^{f}g$ is another metric in the conformal class, then 
$$\theta_{g'}=\theta_g+\di f.$$

 The Weyl structure $D$ is called {\em closed} if $\theta_g$ is closed for one (and thus all) metrics $g\in c$ and {\em exact} if $\theta_g$ is exact for all $g\in c$. From the above formula we see that if $D$ is exact, so that $\theta_g=\di f$ for some $g\in c$, then $\theta_{e^{-f}g}=0$, thus $D$ is the Levi-Civita connection of the metric $e^{-f}g\in c$. 

The manifold $(M,c,D)$ is called {\em Weyl-reducible} if the Weyl structure $D$ is reducible and non-flat. 

Based on some evidence given by the Gallot theorem on Riemannian cones \cite{gal}, it was conjectured in \cite{bm} that every closed, non-exact Weyl structure on a compact conformal manifold is either irreducible or flat. Matveev and Nikolayevsky \cite{mn} constructed a counterexample to the general conjecture, but later on Kourganoff proved that a weaker form of this conjecture holds: 

\begin{thm}\label{k1} (cf. \cite[Thm. 1.5]{k}). A closed non-exact Weyl structure $D$ on a compact conformal manifold $M$, is either flat or irreducible, or the universal cover $\widetilde M$ of $M$ together with the Riemannian metric $g_D$ whose Levi-Civita connection is $D$, is the Riemannian product of a complete flat space $\mathbb{R}^q$ and an incomplete Riemannian manifold $(N,g_N)$ with irreducible holonomy:
$$(\widetilde M,g_D)=\mathbb{R}^q\times (N,g_N).$$
\end{thm}

In \cite[Example~1.6]{k} (see also \cite{mn}), examples of closed reducible Weyl structures on compact manifolds are constructed using a linear map $A\in \mathrm{SL}_{q+1}(\mZ)$, such that:
\begin{enumerate}
\item there exists an $A$-invariant decomposition $\mR^{q+1}=E^s\oplus E^u$ with $\mathrm{dim}(E^u)=1$ and $A|_{E^u}=\lambda ^q\mathrm{Id}_{E^u}$ for some real number $\lambda>1$;
\item there exists a positive definite symmetric bilinear form $b$ on $E^s$, such that $\lambda A|_{E^s}$ is orthogonal with respect to $b$. 
\end{enumerate}
Then $A$ induces a diffeomorphism (also denoted by $A$) of the torus $\mathbb T^{q+1}$, whose mapping torus $M_A:=\mathbb T^{q+1}\times (0,\infty)/(x,t)\sim (Ax,\frac1\lambda t)$, carries a reducible non-flat closed Weyl structure $D_\varphi$ obtained by projecting to $M_A$ the Levi-Civita connection of the metric on  $\mathbb T^{q+1}\times (0,\infty)$ given by:
\[g_\varphi:=\di x_1^2+\cdots +\di x_q^2+\varphi(t)\di x_{q+1}^2+\di t^2,\]
where $x_1, \dots, x_{q+1}$ are the local coordinates with respect to an orthonormal basis $(e_1, \dots, e_{q+1})$ with $e_1, \dots, e_{q}\in E^s$, $e_{q+1}\in E^u$, and $\varphi\colon (0,+\infty)\to (0,+\infty)$ is any smooth function satisfying $\varphi(\lambda t)=\lambda^{2q+2}\varphi(t)$ for every $t\in(0,+\infty)$.

Moreover, Kourganoff proved that these are, up to diffeomorphism, the only examples of Weyl-reducible manifolds when the incomplete factor $N$ is 2-dimensional:

\begin{thm} \label{k2} \cite[Theorem~1.7]{k} Assume that $D$ is a closed non-exact Weyl structure $D$ on a compact conformal manifold $M$ which is neither flat nor irreducible. If the irreducible manifold $N$ given by Theorem \ref{k1} is $2$-dimensional, then $(M,D)$ is isomorphic to one of the Riemannian manifolds $(M_A,D_\varphi)$.
\end{thm}

It turns out, however, that matrices $A\in \mathrm{SL}_{q+1}(\mZ)$ satisfying the conditions (1) and (2) above, only exist for $q=1$ or $q=2$. This is the object of the next section.

 \section{A number-theoretical result}

\begin{prop}\label{q12}
Let $q\in\mN^*$ and $A\in\SL_{q+1}(\mZ)$, such that there is a direct sum decomposition $\mR^{q+1}=E^s\oplus E^u$ invariant by $A$, with $\dim(E^u)=1$. If there exists a positive definite symmetric bilinear form $b$ on $E^s$ and a real number $\lambda>1$, such that $\lambda A|_{E^s}$ is orthogonal with respect to $b$, then $q\in\{1,2\}$.
\end{prop}

\begin{proof}
Let $C$ be a symmetric positive definite matrix, such that $b=\<C^2\cdot,\cdot\>$, where $\<\cdot,\cdot\>$ is the standard Euclidean scalar product. Then the following equivalence holds:\\
$$\lambda A|_{E^s} \in \mathrm{O}(E^s, b) \Longleftrightarrow C\cdot(\lambda A|_{E^s})\cdot C^{-1}\in \mathrm{O}(q).$$
In particular, each eigenvalue of $ \mathrm{Spec(\lambda A|_{E^s})}$ has modulus $1$ and the characteristic polynomial of $A$ denoted by $\mu_A$ is given by:
  $$\mu_A(X)=(X-\lambda^q)\prod_{j=1}^{q}\left(X-\frac{z_j}{\lambda}\right),$$
where $z_j$ are complex numbers with $|z_j|=1$ for all $j\in\{1,\ldots,q\}$, and $\prod_{j=1}^q z_j=1$. Note that $\mu_A$ is irreducible in $\mathbb{Z}[X]$, since if it were a product of two non-constant polynomials with integer coefficients, one of them would have all roots of modulus less than $1$, which is impossible.
We distinguish the following two cases:

{\bf Case 1.} If $q=2p$ is even, denoting $\mu_A(X)=\sum_{j=0}^{2p+1} a_j X^j$ with $a_j\in\mZ$ and $a_{2p+1}=1$, $a_0=-1$, we get 
$$\lambda^{2p}+\frac1\lambda\sum_{j=1}^{2p}z_j=-a_{2p},\qquad \lambda^{-2p}+\lambda\sum_{j=1}^{2p}\frac1{z_j}=a_{1}.$$
This shows that the sum $s:=\sum_{j=1}^{2p}z_j$ is real, and since $|z_j|=1$ for all $j\in\{1,\ldots,2p\}$, $s$ is also equal to $\sum_{j=1}^{2p}\frac1{z_j}$. Eliminating $s$ from the two equations above, yields
$$\lambda^{4p+2}+a_{2p}\lambda^{2p+2}+a_1\lambda^{2p}-1=0.$$
Consequently, $\lambda^2$ is root of the polynomial $$Q(X):=X^{2p+1}+a_{2p}X^{p+1}+a_1X^{p}-1.$$ 
Denote by $r_1,\ldots,r_{2p}$ the other complex roots of $Q$. Newton's relations show that there exists a monic polynomial $\widetilde Q\in\mZ[X]$ whose roots are $\lambda^{2p},r_1^p,\ldots,r_{2p}^p$. The monic polynomials $\mu_A$ and $\widetilde Q\in \mZ[X]$ have both degree ${2p+1}$ and $\lambda^{2p}$ is a common root. Since $\mu_A$ is irreducible, they must coincide, so up to a permutation, one can assume that $r_j^p=\frac{z_j}{\lambda}$ for all $j\in\{1,\ldots,2p\}$. This shows that $\lambda^{\frac1p}r_j$ are complex numbers of modulus one for all $j\in\{1,\ldots,2p\}$. 

If $p\ge 2$, the coefficients of $X^{2p}$ and $X$ in the polynomial $Q$ vanish, so 
$$\lambda^2+\sum_{j=1}^{2p}r_j=0=\frac1{\lambda^2}+\sum_{j=1}^{2p}\frac1{r_j}.$$
Thus $\sum_{j=1}^{2p}r_j=-\lambda^2$ and as $|\lambda^{\frac1p}r_j|=1$ for all $j$, 
$$-\lambda^{-2}=\sum_{j=1}^{2p}\frac1{r_j}=\lambda^{\frac2p}\sum_{j=1}^{2p}r_j=-\lambda^{\frac2p}\lambda^2.$$
This contradicts the fact that $\lambda>1$, showing that $p=1$ and therefore $q=2$ (see also \cite[Lemma 3.5]{ao}).

{\bf Case 2.} If $q$ is odd, then $\mu_A$ has at least one further real root, so either $\frac{1}{\lambda}$ or $-\frac{1}{\lambda}$ is a root of $\mu_A$. Up to reordering the subscripts one thus has $z_1=\pm 1$. Assume that $z_1=1$ (the argument for $z_1=-1$ is the same). The monic polynomial $P\in\mZ[X]$ defined by $P(X):=X^{q+1}\mu_A(\frac{1}{X})$ satisfies $P(0)=1$, and its roots are $\{\lambda^{-q}, \lambda, \frac{\lambda}{z_2}, \dots, \frac{\lambda}{z_q}\}$.

By Newton's identities again, there exists a monic polynomial $\widetilde P\in\mZ[X]$ with $\widetilde P(0)=1$, whose roots are $\{\lambda^{-q^2}, \lambda^q, (\frac{\lambda}{z_2})^q, \dots, (\frac{\lambda}{z_q})^q\}$.

Since the monic polynomials $\mu_A$ and $\widetilde P\in \mZ[X]$ (of same degree) have $\lambda^q$ as common root, and $\mu_A$ is irreducible, they must coincide. In particular $\lambda^{-q^2}$ is a root of $\mu_A$. On the other hand every root of $\mu_A$ has complex modulus equal to either $\lambda^q$ or $\frac1\lambda$. Since $\lambda>1$, we obtain $q=1$.
\end{proof}

\begin{remark}
As pointed out by V. Vuletescu, for odd $q$, Proposition \ref{q12} also follows from a more general result of Ferguson \cite{f}, whose proof, however, is rather involved.
\end{remark}

\section{Applications}

Our main application concerns locally conformally Kähler manifolds. Recall that a Hermitian manifold $(M,g,J)$ of complex dimension $n\geq 2$ is called {\em locally conformally K\"ahler} (in short, lcK) if around every point in $M$ the metric $g$ can be conformally rescaled to a K\"ahler metric. This condition is equivalent to the existence of a closed 1-form $\theta$, such that 
\begin{equation*}\label{eqomega}
\dd\Omega=\theta\wedge \Omega,\end{equation*}
where $\Omega:=g(J\cdot,\cdot)$ denotes the fundamental 2-form. 
Let now $\widetilde M$ be the universal cover of an lcK manifold $(M,J,g,\theta)$, endowed with the pull-back lcK structure $(\tilde J,\tilde g,\tilde \theta)$. Since $\widetilde M$ is simply connected, $\ti\theta$ is exact, {\em i.e.} $\tilde\theta=\di\varphi$, and by the above considerations, the metric $g^K:=e^{-\varphi}\tilde g$ is K\"ahler. 

The group $\pi_1(M)$ acts on $(\widetilde M,\tilde J,g^K)$ by holomorphic homotheties. Furthermore, we assume that the lcK structure is strict, in the sense that $\pi_1(M)$ is not a subgroup of the isometry group of $(\widetilde{M}, g^K)$. In particular, the Levi-Civita connection of the K\"ahler metric $g^K$ projects to a closed, non-exact Weyl structure on ${M}$, called the {\em standard Weyl structure}. Its Lee form with respect to $g$ is exactly $\theta$.

Due to the fact that the real dimension of an lcK manifold is even, applying Proposition \ref{q12} to the special case of a compact strict lcK manifold whose standard Weyl structure is reducible, we obtain the following:

\begin{prop}\label{q2}
Let $M$ be a compact Weyl-reducible strict lcK manifold. If the irreducible factor $N$ in the splitting of the universal cover  $(\widetilde{M}, g^K)$ as a Riemannian product $\mR^q\times N$ given by Theorem \ref{k1}
is $2$-dimensional, then $q=2$ and thus $M$ is an Inoue surface $S^0$, cf.  \cite{i}.
\end{prop}

Let us remark that if in Proposition~\ref{q2} we drop the assumption on the dimension of the irreducible factor, then there are many more examples of Weyl-reducible lcK structures. They are obtained on lcK manifolds constructed by Oeljeklaus and Toma \cite{ot}, for every integer $s\ge 1$, on certain compact quotients $M_{\Gamma}$ of $\mC\times \mH^s$, where $\mH$ denotes the upper complex half-plane, $\Gamma$ are certain groups whose action on $\mC\times \mH^s$ is cocompact and properly discontinuous  (for the precise definition of $\Gamma$ and its action see \cite{ot}). We will briefly review them here. 

In order to define the lcK structure on the quotient $M_{\Gamma}$, Oeljeklaus and Toma consider the function
$$F:\mC\times \mH^s\to \mR,\qquad F(z,z_1,\ldots,z_s):=|z|^2+\frac1{y_1\ldots y_s},$$
with $z_k=x_k+iy_k$ and claim that it is a global Kähler potential on $\mC\times \mH^s$ (note a small sign error in \cite{ot}). To check this, we introduce
$$u: \mH^s\to \mR,\qquad u(z_1,\ldots,z_s):=\frac1{y_1\ldots y_s}=\frac{(2i)^s}{\prod_{j=1}^s(z_j-\bar z_j)},$$
and compute 
\begin{equation}\label{e1}\bar\partial u=u\sum_{j=1}^s\frac {\di \bar z_j}{z_j-\bar z_j},\qquad \partial u=-u\sum_{j=1}^s\frac {\di z_j}{z_j-\bar z_j}, \end{equation}
\begin{eqnarray*}\partial \bar\partial u&=&\partial u\wedge\sum_{j=1}^s\frac {\di \bar z_j}{z_j-\bar z_j}-u\sum_{j=1}^s\frac {\di z_j\wedge\di \bar z_j}{(z_j-\bar z_j)^2}\\&=&-u\sum_{j,k=1}^s\frac {1+\delta_{jk}}{(z_j-\bar z_j)(z_k-\bar z_k)}\di z_j\wedge\di \bar z_k, \end{eqnarray*}
whence 
\begin{equation}\label{e2}\partial \bar\partial u=\frac{u}{4}\sum_{j,k=1}^s\frac {1+\delta_{jk}}{y_jy_k}\di z_j\wedge\di \bar z_k.\end{equation}

This shows that $i\partial \bar\partial u$ is the fundamental 2-form of a Kähler metric $h$ on $\mH^s$ whose coefficients are 
$h_{j\bar k}=\frac{u}{4}\frac {1+\delta_{jk}}{y_jy_k}$. 

\begin{prop} \label{irr} The Kähler metric on $\mH^s$ with Kähler potential $u$ is irreducible.  
\end{prop}
\begin{proof}
The matrix $(h_{j\bar k})$ can be written as the product of 3 matrices
$$(h_{j\bar k})=\frac{u}{4}\begin{pmatrix}\frac{1}{y_1}&0&\ldots &0\\ 0&\frac{1}{y_2}&\ldots &0 \\ \vdots &\vdots &\ddots&\vdots\\ 0&0&\ldots &\frac{1}{y_s}\end{pmatrix}\begin{pmatrix}2&1&\ldots &1\\ 1&2&\ldots &1 \\ \vdots &\vdots &\ddots&\vdots\\1&1&\ldots &2\end{pmatrix}\begin{pmatrix}\frac{1} {y_1}&0&\ldots &0\\ 0&\frac{1}{y_2}&\ldots &0 \\ \vdots &\vdots &\ddots&\vdots\\ 0&0&\ldots &\frac{1}{y_s}\end{pmatrix},$$
so its determinant equals 
$$\det(h_{j\bar k})=\left(\frac{u}{4}\right)^s(s+1)\frac1{(y_1\ldots y_s)^2}=\frac{(s+1)u^{s+2}}{4^s}.$$
The usual formula for the Ricci form $\rho$ of $h$ (cf. e.g. \cite[Eq. (12.6)]{kg}) together with \eqref{e1} and \eqref{e2} gives
\begin{eqnarray*}\rho&=&-i\partial \bar\partial\ln(\det(h_{j\bar k}))=-i(s+2)\partial \bar\partial\ln(u)=-i(s+2)\partial (\frac1u\bar\partial u)\\
&=&-i(s+2)\left(\frac1u\partial\bar\partial u-\frac1{u^2}\partial u\wedge\bar\partial u\right)\\&=&-\frac{i(s+2)}4\sum_{j,k=1}^s\frac {2+\delta_{jk}}{y_jy_k}\di z_j\wedge\di \bar z_k.
\end{eqnarray*}
This shows that the Ricci tensor of $h$ is negative definite on $\mH^s$, so $h$ is irreducible. 
\end{proof}

As a consequence of Proposition \ref{irr}, the Kähler metric on $\mC\times\mH^s$ with fundamental 2-form $\Omega=i\partial\bar\partial F=i\di z\wedge\di\bar z+i\partial\bar\partial u$ is the product of the flat metric on $\mC$ with an irreducible Kähler metric on $\mH^s$. Therefore, the induced lcK structure on the compact quotient $M_{\Gamma}$ is Weyl-reducible, and the irreducible factor of the universal cover given by Theorem \ref{k1} is exactly $N=\mH^s$, so it has dimension $2s$.

{\sc Acknowledgments}. This work has been partially supported by the Procope Project No. 42513ZJ (France)/ 57445459 (Germany). We would like to thank Victor Vuletescu for having pointed out a small gap in the original proof of Proposition \ref{q12}.

\end{document}